\documentclass[11pt]{article}
\usepackage{amsfonts}
\usepackage{graphicx}
\usepackage{amsmath}
\usepackage{amsthm}
\usepackage{amssymb}
\usepackage{amscd}
\usepackage{graphicx}
\usepackage{microtype}
\usepackage{enumerate}
\usepackage{fullpage}
\usepackage{pinlabel}
\usepackage[top=1.3in,bottom=1.4in,left=1.1in,right=1.1in]{geometry}
\usepackage{url}

\theoremstyle{definition}
\newtheorem{theorem}{Theorem}[section]

\newtheorem{claim}[theorem]{Claim}

\newtheorem{corollary}[theorem]{Corollary}
\newtheorem{definition}[theorem]{Definition}

\newtheorem{fact}[theorem]{Fact}
\newtheorem{lemma}[theorem]{Lemma}

\newtheorem{proposition}[theorem]{Proposition}
\newtheorem{remark}[theorem]{Remark}

\newcommand{\R}{\mathbb{R}}

\newcommand{\Z}{\mathbb{Z}}

\newcommand{\boldhead}[1]{%
 {\bigskip \noindent \bfseries #1 \\ }}

\newcommand{\Agerm}{\mathcal{A}}
\newcommand{\germ}{\mathcal{G}_\infty}

\DeclareMathOperator{\supp}{supp}

\DeclareMathOperator{\Homeo}{Homeo}
\DeclareMathOperator{\fix}{fix}

\DeclareMathOperator{\Aff}{Aff_+(\mathbb{R})}
\DeclareMathOperator{\id}{id}

\title{\vspace{-1cm} Left-orderable groups that don't act on the line
}
\author{Kathryn Mann}
\date{}

\begin{document}

\maketitle  

\abstract{We show that the group $\germ$ of germs at infinity of orientation-preserving homeomorphisms of $\R$ admits no action on the line.  This gives an example of a left-orderable group of the same cardinality as $\Homeo_+(\R)$ that does not embed in $\Homeo_+(\R)$.    
As an application of our techniques, we construct a finitely generated group $\Gamma \subset \germ$ that does not extend to $\Homeo_+(\R)$ and, separately, extend a theorem of E. Militon on homomorphisms between groups of homeomorphisms.  
}

\section{Introduction}

\begin{definition}
A group $G$ is \emph{left-orderable} if there is a total order $\leq$ on $G$ that is invariant under left multiplication. 
\end{definition}
The study of left-orderable groups and left-invariant orders on groups has deep connections with algebra, dynamics, and topology. Examples of left-orderable groups include all torsion-free abelian groups, free groups, braid groups, the group $\Homeo_+(\R)$ of orientation-preserving homeomorphisms of the line, and the fundamental groups of orientable surfaces.  We refer the reader to \cite{DNR} for an introduction to the subject from a dynamical viewpoint.     

One important link between orders and dynamics comes from the following classical theorem (in \cite{DNR} this theorem is attributed to \cite{Holland}) relating left-invariant orders to actions on the line.  

\begin{theorem}[see Theorem 6.8 in \cite{Ghys} for a proof] \label{realization thm}
Let $G$ be a countable group.  Then $G$ is left-orderable if and only if there is an injective homomorphism $G \to \Homeo_+(\R)$.  Moreover, given an order on $G$, there is a canonical (up to conjugacy in $\Homeo_+(\R)$) injective homomorphism $G \to \Homeo_+(\R)$ called a  \emph{dynamical realization}.
\end{theorem}

Theorem \ref{realization thm} does not apply to uncountable groups.  In particular, a free abelian group of cardinality larger than $|\R|$ is left-orderable, but obviously cannot embed in $\Homeo_+(\R)$, which has cardinality equal to $|\R|$.    However, there are also uncountable, left-orderable groups that \emph{do} embed in $\Homeo_+(\R)$ -- one example is $\Homeo_+(\R)$ itself.  

Remarkably, there seem to be very few known examples of uncountable left ordered groups of cardinality $|\R|$ that don't act on the line.  One method to construct examples is to take a group $\Gamma$ that has only finitely many left orders (and hence strong constraints on its actions on the line), and build a group $G$ containing uncountably many copies of $\Gamma$ related to each other in an appropriate way.  We conclude this paper with two examples that illustrate this method; the main one is due to C. Rivas.

The central result of this paper provides an interesting complementary example -- a naturally occurring group of cardinality $|\R|$  that has no dynamical realization.

\begin{definition}
The \emph{group of germs at $\infty$} of homeomorphisms of $\R$, denoted $\germ$, is the set of equivalence classes of orientation-preserving homeomorphisms under the equivalence relation $f \sim g$ if $f$ and $g$ agree on some neighborhood $[x, \infty)$ of $\infty$.   
Composition of homeomorphisms descends from $\Homeo_+(\R)$ to $\germ$, making $\germ$ a group.  
\end{definition}

Navas has shown that $\germ$ is left-orderable (see Proposition \ref{lo prop} below).  
Our main theorem is the following.  

\begin{theorem} \label{main thm}
There is no nontrivial homomorphism $\germ \to \Homeo_+(\R)$.  
\end{theorem}

As a consequence, we have
\begin{corollary} \label{cardinality cor}
There exists a left-orderable group with cardinality equal to that of $\Homeo_+(\R)$ that does not embed in $\Homeo_+(\R)$.  
\end{corollary}

\begin{proof}[Proof of Corollary \ref{cardinality cor}]
By the remarks above, we need only show that $|\germ| = |\R|$.  The natural map $\Homeo_+(\R) \to \germ$ is a surjection.  We can define an injection (in fact an injective homomorphism) $\phi: \Homeo_+(\R) \to \germ$ as follows.  For each $n \in \Z$, and each interval $(n,n+1) \subset \R$, let $i_n: \Homeo_+(\R) \to \Homeo_+(n,n+1)$ be a homeomorphism, and define $\phi(f)$ by 
$$\phi(f)(x) = i_n(f)(x) \, \text { for } x \in (n, n+1).$$    
\end{proof}

\boldhead{Extension vs. realization}
A left-invariant order on a group $G$ induces a left-invariant order on any subgroup of $G$ in a natural way.  Thus, Theorem \ref{realization thm} implies that any \emph{countable} subgroup $\Gamma$ of a left-orderable group $G$ has a dynamical realization whose dynamical properties depend only on the order on $G$.   In this sense, dynamical realizations of subgroups tell us about the order on a group.  

Navas' proof that $\germ$ is orderable (Proposition \ref{lo prop}) is not constructive, so we do not know what a left-invariant order on $\germ$ might look like, or what a dynamical realization of a subgroup might look like.   To address this, Navas asked in particular whether there is an obstruction to realizing a subgroup $\Gamma \subset \germ$ in $\Homeo_+(\R)$ by \emph{extending} it to $\Homeo_+(\R)$ -- giving a homomorphism $\Phi: \Gamma \to \Homeo_+(\R)$ such that  the composition 
$$\Gamma \overset{\Phi}\to \Homeo_+(\R) \overset{\pi}\to \germ$$ is the identity on $\Gamma$.   (Here, and in what follows, $\pi$ denotes the natural map from $\Homeo_+(\R)$ to $\germ$).  

As an application of our techniques, we give a negative answer to Navas' question. 

\begin{proposition} \label{extension prop} 
There exists a finitely generated group $\Gamma \subset \germ$ that admits no extension to $\Homeo_+(\R)$.  
\end{proposition}
This group is described explicitly in Section \ref{applications sec}.

\boldhead{Further applications}
In Section \ref{applications sec} we also show that $\germ$ does not act on the circle, and use this to extend a theorem of E. Militon on actions of groups of homeomorphisms on 1-manifolds.  

\boldhead{Acknowledgements}
I thank Andr\'es Navas for introducing me to this problem and much of the background material.  I would also like to thank the University of Santiago for its hospitality during my visit in April 2014, and Sebastian Hurtado, Andres Navas, and Crist\'obal Rivas for many productive conversations during the visit (which inspired this work) and their contributions to this paper.  Finally, I thank Danny Calegari, Benson Farb, John Franks, and Dave Witte-Morris for their interest in this project and Emmanuel Militon for helpful feedback, especially with Section \ref{applications sec}.

\section{Properties of $\germ$}
 
In this section we introduce basic properties of $\germ$ and the main tools used in the proof of Theorem \ref{main thm}.  
In addition to showing that $\germ$ is left-orderable, we will show that it is a simple group so any nontrivial homomorphism $\germ \to \Homeo_+(\R)$ is necessarily injective.  
The section concludes with a proof of a ``warm-up" theorem (Proposition \ref{iso prop} below) illustrating some key ideas used in the proof of Theorem \ref{main thm}.   

\subsection{Left-orderability}

We begin with Navas' proof that $\germ$ is left-orderable.  It uses the following well known criterion for left-orderability.  

\begin{proposition} \label{lo criterion}
A group $G$ is left-orderable if and only if, 
for every finite collection of nontrivial elements $g_1, ... , g_k$, there exist choices $\epsilon_i \in \{-1, 1\}$ such that the identity is not an element of the semigroup generated by $\{g_i^{\epsilon_i}\}$.  
\end{proposition}

It is obvious that this condition is necessary -- if $G$ is left orderable, then we can choose $\epsilon_i \in \{-1, 1\}$ such that $g_i^{\epsilon_i} > \id$ holds for each $i$, and this will satisfy the requirement above.  It is a bit more work to show the condition is sufficient; we refer the reader to Prop. 1.4 of \cite{Navas} for a proof.

\begin{proposition}[Navas]  \label{lo prop}
$\germ$ is left-orderable.  
\end{proposition}

\begin{proof}
We use the criterion in Proposition \ref{lo criterion}.  Let $\{ g_1, g_2, ... ,g_k \}$ be a finite subset of nontrivial elements of $\germ$, and choose homeomorphisms $f_1, ... ,f_k$ such that the germ of $f_i$ is $g_i$.  

Let $\{ x_{1,n} \}$ be a sequence of points with $\lim \limits_{n \to \infty} x_{1,n} = \infty$, and such that no point $x_{1,n}$ is fixed by every homeomorphism $f_i$. 
After passing to a subsequence, we may assume for each of the $i$ that either $f_i(x_{1,n}) > x_{1,n}$ holds for all $n$, or $f_i(x_{1,n}) < x_{1,n}$ holds for all $n$, or $f_i(x_{1,n}) = x_{1,n}$ holds for all $n$.  
In the first case we let $\epsilon_i = +1$, in the second let $\epsilon_i = -1$, and in the third leave $\epsilon_i$ undefined.  Note that the condition that no point $x_{1,n}$ was fixed by every $f_i$ implies that we have defined at least one $\epsilon_i$.  

Provided some $\epsilon_i$ are still undefined, consider the set of $f_i$ for which $\epsilon_i$ is undefined, and repeat the procedure described above for these homeomorphisms -- take a sequence $\{ x_{2,n} \}$ with $\lim \limits_{n \to \infty} x_{2,n} = \infty$ such that no point is fixed by each of these $f_i$, pass to a subsequence as above, and define $\epsilon_i$ depending on whether $f_i(x_{2,n}) > x_{2,n}$ holds for all $n$, or $f_i(x_{2,n}) < x_{2,n}$ holds for all $n$.  If for some $i$,  $f_i(x_{2,n}) = x_{2,n}$ holds for all $n$, leave these $\epsilon_i$ undefined, and repeat the procedure again.   The process terminates after at most $k$ steps.

Note that, by construction, $f_i^{\epsilon_i}(x_{j,n}) \geq x_{j,n}$ for all $i, j$ and $n$.  Moreover, for each $i$ there exists $j$ such that $f_i^{\epsilon_i}(x_{j,n}) > x_{j,n}$ holds for \emph{all} $n$.  This implies that, for any word $f$ in the semigroup generated by $\{f_i^{\epsilon_i}\}$, there exists $j$ such that $f(x_{j,n}) > x_{j,n}$ for all $n$.  Since $\lim \limits_{n \to \infty} x_{j,n} = \infty$, the germ of $f$ is nontrivial.  
\end{proof}

\subsection{Simplicity} 
 
Our next goal is to prove the following.
\begin{proposition} \label{simple prop}
$\germ$ is a simple group.
\end{proposition}
This result is essentially due to Fine and Schweigert \cite{FS}, who give a complete classification of all normal subgroups of $\Homeo_+(\R)$. 
Since we do not need the full classification, we'll give a much shorter, self-contained proof that $\germ$ is simple here.  Our proof builds on the following elementary fact.

\begin{fact} \label{basic fact}
Any pair of homeomorphisms $f_1, f_2 \in \Homeo_+[0,1]$ satisfying 
$$f_i(x) > x \text{  for all  } x \in (0,1)$$
are conjugate in $\Homeo_+[0,1]$.
\end{fact} 

Germs with the simplest possible dynamics are \emph{fixed point free}. 
\begin{definition} A germ $g \in \germ$ is \emph{fixed point free} if there exists a homeomorphism $f$ with germ $g$, and an interval $[x, \infty)$ such that $f(y) \neq y$ for all $y \in [x, \infty)$.
\end{definition}

\noindent It is a consequence of Fact \ref{basic fact} there are precisely two conjugacy classes of fixed point free germs: those that have representative homeomorphisms that are strictly increasing on some neighborhood of $\infty$, and those with representatives that are strictly decreasing on some neighborhood of $\infty$.  

\bigskip
Using fixed point free germs, we now prove that $\germ$ is simple.  

\begin{proof}[Proof of Proposition \ref{simple prop}]

Suppose $\mathcal{N} \subset \germ$ is a nontrivial normal subgroup. 

\begin{lemma} \label{fix point free lemma}
$\mathcal{N}$ contains a fixed point free germ.  
\end{lemma}

\begin{proof}
Let $h$ be a homeomorphism with germ a nontrivial element of $\mathcal{N}$.  Then (perhaps after replacing $h$ with its inverse) there is a sequence of points $x_1, x_2, x_3, ... $ with $\lim \limits_{n \to \infty} x_n = \infty$ and such that $h(x_n) > x_n$.  After passing to a subsequence if necessary, we can also assume that $h(x_n) < x_{n+1}$.  Let $g \in \Homeo_+(\R)$ be a homeomorphism such that $g(x_n) = h(x_n)$ and $g(h(x_n)) = x_{n+1}$ holds for each $n$.  We claim that $h ghg^{-1}$ has fixed point free germ at infinity -- in fact, we will show that $h ghg^{-1}(x) > x$ for all $x \geq x_1$.  
By construction,
$$h ghg^{-1}(h(x_n)) = hghg^{-1}(g(x_n)) = hgh(x_n) = h(x_{n+1}),$$
so $h ghg^{-1} \big( \left[ h(x_n), \, h(x_{n+1})\right) \big) = \left[h(x_{n+1}), \,h(x_{n+2}) \right)$, which shows that $h ghg^{-1}(x) > x$ for $x \geq x_1$.
\end{proof}

Since all fixed point free germs are conjugate either to $h ghg^{-1}$ or its inverse, it follows that $\mathcal{N}$ contains \emph{all} fixed point free germs.  
Now we can easily show that $\mathcal{N} = \germ$.  Let $f$ be any homeomorphism of $\R$.   Let $f_2$ be defined on $[0, \infty)$ by 
$$f_2(x) = \max\{f^{-1}(x) + 1, x+1\} \text{ for }x \in [0, \infty).$$   Then $f_2$ can be extended to a homeomorphism $\R \to \R$, and will satisfy $f_2(x) > x $ for all $x > 0$ and $f_2 f(x) > x $ for all $x>0$.  Thus, the germs of both $f_2$ and $f_2 f$ are fixed point free and lie in $\mathcal{N}$,
so the germ of $f$ lies in $\mathcal{N}$ as well, which is what we needed to show.  

\end{proof}

\subsection{A warm-up theorem: $\germ \ncong \Homeo_c(\R)$}
As a warm-up to the proof of Theorem \ref{main thm}, and to introduce some important techniques, we give a short proof of the following strictly weaker result.  Recall that $\Homeo_c(\R)$ denotes the group of homeomorphisms with compact support. 

\begin{proposition} \label{iso prop}
$\germ$ is not isomorphic to $\Homeo_c(\R)$.  
\end{proposition}

\begin{remark}  \label{simple rk}
It is clear that $\germ$ is not isomorphic to $\Homeo_+(\R)$, since $\germ$ is simple and $\Homeo_+(\R)$ is not simple -- in fact $\Homeo_c(\R)\subset \Homeo_+(\R)$ is a normal subgroup.   However, $\Homeo_c(\R)$ \emph{is} a simple group, so simplicity provides no obstruction to an isomorphism.    
Proving simplicity of $\Homeo_c(\R)$ is actually not too difficult -- a nice exposition (for the case of $\Homeo_+(S^1)$, but the $\Homeo_c(\R)$ case is analogous) can be found in \cite{Ghys}.
\end{remark}

To prove Proposition \ref{iso prop} we will look at the actions of a particular subgroup, $\Homeo_\Z(\R)$.  This group also plays an important role in the proof of Theorem \ref{main thm}.

\begin{definition}
Let $T$ denote the translation $x \mapsto x +1$.  The group $\Homeo_\Z(\R)$ is the centralizer of $T$ in $\Homeo_+(\R)$. 
\end{definition}

The reader may notice that a group quite similar to $\Homeo_\Z(\R)$ has already made an appearance in Corollary \ref{cardinality cor}.   More precisely, let $H_\Z \subset \Homeo_\Z(\R)$ be the subgroup consisting of homeomorphisms that pointwise fix the integers.  Then $H_\Z$ is naturally isomorphic to $\Homeo_+(\R)$, and the natural map $\Homeo_+(\R) \cong H_\Z \overset{\pi}\to \germ$ is an example of an injective homomorphism just as described in the proof of Corollary \ref{cardinality cor}.  

\bigskip
The key to our proof of Proposition \ref{iso prop} (and also of Theorem \ref{main thm}) is a lemma of Militon, which states that all actions of $\Homeo_\Z(\R)$ on the line have a standard form.  We call this form \emph{topologically diagonal}.  

\begin{definition}
A \emph{topologically diagonal embedding} of a group $G \subset \Homeo_+(\R)$ is a homomorphism $\phi: G \to \Homeo_+(\R)$ defined as follows.   Choose a collection of disjoint open intervals $I_n \subset \R$ and homeomorphisms $f_n: \R \to I_n$.  Define $\phi$ by
$$\phi(g)(x) =	\left\{ \begin{array}{rl} f_n g f_n^{-1}(x) & \mbox{if $x \in I_n$}
 \\ x & \mbox{otherwise} 
\end{array}\right.
$$
\end{definition}

\begin{lemma}[Militon; Lemma 5.1 in \cite{Militon}] \label{militon lemma}
Let $\phi: \Homeo_\Z(\R) \to \Homeo_+(\R)$ be a nontrivial homomorphism. Then $\phi$ is a topologically diagonal embedding.  
\end{lemma} 

The proof of Militon's lemma is not difficult, although it uses one deeper result of Matsumoto \cite{Matsumoto}.  We give a short version of Militon's proof for the convenience of the reader.   Matsumoto's result (Theorem 5.3 in \cite{Matsumoto}) is that any homomorphism $\Homeo_+(S^1) \to \Homeo_+(S^1)$ is given by conjugation by an element of $\Homeo_+(S^1)$; the reasons for this are essentially cohomological.  

\begin{proof}[Proof of Lemma \ref{militon lemma}]
Let $\phi: \Homeo_\Z(\R) \to \Homeo_+(\R)$ be a homomorphism, and consider the set of points fixed by $\phi(T)$.  
If $\fix(\phi(T)) = \emptyset$, then Fact \ref{basic fact} implies that $T$ is conjugate to a translation.  Thus, $\R / \langle T \rangle = S^1$, and $\Homeo_\Z(\R) / \langle T \rangle \cong \Homeo_+(S^1)$ acts on $\R / \langle T \rangle$ by homeomorphisms.   By Matsumoto's result, this action comes from conjugation by a homeomorphism of $\R / \langle T \rangle$, which will lift to a homeomorphism $f: \R \to \R$ such that $\phi(g) = fgf^{-1}$ for all $g \in \Homeo_\Z(\R)$.  

Now suppose $\fix(\phi(T)) \neq \emptyset$. Using the case above, it suffices to show each point of $\fix(\phi(T))$ is a global fixed point for $\phi(\Homeo_\Z(\R))$.  
Since $T$ is central, $\fix(\phi(T))$ is preserved by $\phi(\Homeo_+(\R))$.  Thus, we get an induced action of $\Homeo_\Z(\R) / \langle T \rangle \cong \Homeo_+(S^1)$  on $\fix(\phi(T))$, and this action preserves the natural (linear) order on $\fix(\phi(T))$ inherited from $\R$.  It follows that finite order elements of $\Homeo_+(S^1)$ act trivially on $\fix(\phi(T))$.  Since $\Homeo_+(S^1)$ is simple (as noted in Remark \ref{simple rk} above), its action on $\fix(\phi(T))$ must be trivial, and this is what we needed to show.  

\end{proof}

Before proving Proposition \ref{iso prop}, we need one more easy lemma. 
\begin{lemma}\label{centralizer lemma}
Suppose that $g \in \germ$ is a germ that commutes with all germs of homeomorphisms in $\Homeo_\Z(\R)$.  Then $g$ is the germ of an element of $\Homeo_\Z(\R)$
\end{lemma}

In fact, one can probably show under this hypothesis that $g$ is the germ of the translation $T$, but we won't need this stronger fact.  %

\begin{proof}[Proof of Lemma \ref{centralizer lemma}]
Suppose $g$ is a germ that commutes with all germs of elements of $\Homeo_\Z(\R)$.  Then $g$ commutes with the germ of $T$.  Let $f$ be any homeomorphism with germ $g$.  Then $[f, T]$ is the identity on some neighborhood of $\infty$, so $f$ commutes with $T$ on a neighborhood of $\infty$.  It follows that the restriction of $f$ to this neighborhood agrees with an element of $\Homeo_\Z(\R)$ and so $g$ is the germ of an element of $\Homeo_\Z(\R)$. 
\end{proof}

With these tools, we can now easily prove that $\germ$ and $\Homeo_c(\R)$ are not isomorphic.  
\begin{proof}[Proof of Proposition \ref{iso prop}]
Suppose for contradiction that $\Phi: \germ \to \Homeo_c(\R)$ is an isomorphism.  Let $t$ be the germ of the translation $T: x \to x+1$. 
Then $\Phi(t)$ has support contained in some compact interval $I$.   Consider the map 
$$\Homeo_\Z(\R) \overset{\pi}\to \germ \overset{\Phi}\to \Homeo_c(\R).$$ 
Let $G \subset \Homeo_c(\R)$ be the image of $\Homeo_\Z(\R)$ under this map.  
By Militon's Lemma \ref{militon lemma}, $G$ is a collection of homeomorphisms with support contained in $I$.  The centralizer of $G$ in $\Homeo_c(\R)$ contains any homeomorphism $f$ that fixes $I$ pointwise, and in particular contains some homeomorphism $f \notin G$. 

Since $\Phi$ is an isomorphism, it follows that the centralizer of $\pi(\Homeo_\Z(\R))$ in $\germ$ contains an element \emph{not} in $\pi(\Homeo_\Z(\R))$.  But this contradicts Lemma \ref{centralizer lemma}.  

\end{proof}

\section{Proof of Theorem \ref{main thm}}

We begin by constructing an \emph{affine subgroup} of germs.  This subgroup will be isomorphic to the standard group of orientation-preserving affine transformations, $\Aff$, but is \emph{not} the image of $\Aff$ under the natural map $\Aff \hookrightarrow \Homeo_+(\R) \overset{\pi}\to \germ$.   
In Proposition \ref{germ only prop}, we will in fact show (in a precise sense) that a subgroup constructed in this manner \emph{cannot} be the image of the standard affine subgroup.   This gives us a concrete ``difference"  between $\germ$ and $\Homeo_+(\R)$ that will help to prove Theorem \ref{main thm}.  

\begin{lemma}[A nonstandard affine subgroup of $\germ$] \label{affine lemma}
Let $a_t \in \germ$ be the germ of the translation $x \mapsto x+t$.  Then there exists a family of germs $b_s \in \germ$, for $s \in \R$ satisfying 
$$ a_t b_s a_t^{-1} = b_{e^t s} $$ 
$$ b_s b_r = b_r b_s$$
$$ b_{n s} = (b_s)^n$$
for all $s \in \R$, $t>0$ and $n \in \Z$. 
\end{lemma}

\begin{remark} \label{aff remark}
Let $\Agerm$ be the group generated by the $a_t$ and $b_s$ of Lemma \ref{affine lemma}.  Define a homomorphism $\psi: \Agerm \to \Aff$ given by
$$\psi(a_t)(x) = e^t x$$
$$\psi(b_s)(x) = x + s$$
The relations in the statement of Lemma \ref{affine lemma} imply that $\psi$ is a homomorphism.  On the specific group $\Agerm$ constructed in the proof below, $\psi$ will be an \emph{isomorphism}.
\end{remark}

\begin{proof}[Proof of Lemma \ref{affine lemma}]
Let $s \in \R$.  Define $B_s$ on $[\log(|s| + 1), \infty)$ by 
$$B_s(x) = \log(e^x + s)$$
This is an orientation-preserving homeomorphism from $[\log(|s| + 1), \infty)$ to $[\log(|s| + s + 1), \infty)$, so can be extended to an orientation-preserving homeomorphism of $\R$.  Abusing notation, we let $B_s$ denote some such extension, and let $b_s$ be the germ at infinity of $B_s$. 
 
Let $A_t$ denote the translation $x \mapsto x+t$.   Then, for all $x$ in a neighborhood of $\infty$, we have
$$B_r B_s (x) = \log(e^{\log(e^x + s)} + r) = \log(e^{x} + r + s) = B_s B_r (x),$$ 
$$A_t B_s A_t^{-1}(x) = \log(e^{x-t} + s) + t = \log(e^{x-t} + s) + \log(e^t) = \log(e^x + e^t s) = B_{e^ts}(x),$$
and
$$B_{ns}(x) = \log(e^x + ns) = (B_s)^n(x).$$
\end{proof} 

\bigskip
\noindent Our next proposition shows that the construction in Lemma \ref{affine lemma} only works on the level of germs.
\begin{proposition} \label{germ only prop}
Let $A_{t}$ denote translation by $t$.  
There does not exist a collection of globally defined, nontrivial homeomorphisms $B_s \in \Homeo_+(\R)$ such that the conditions 
$$ A_t B_s A_t^{-1} = B_{e^t s} $$ 
$$ B_s B_r = B_r B_s$$
$$ B_{n s} = (B_s)^n$$
hold for all $s \in \R$, $t>0$ and $n \in \Z$.
\end{proposition}

\begin{proof}
Suppose we had such a collection of homeomorphisms.   As a first case, assume that for some $s \in \R$, the homeomorphism $B_s$ acts freely (i.e. without fixed points) on $\R$.  Then $B_s$ is conjugate to the translation $T: x \mapsto x+1$.   It is easy to show, using a Banach contraction principle argument ($A^{-1}$ takes intervals of length $n$ to intervals of length 1), that  any homeomorphism $A$ satisfying $ATA^{-1} = T^n$ must act with a fixed point on $\R$.  

In particular, that
$A_{\log(n)} B_s A_{\log(n)}^{-1}(x) = (B_{s})^n(x)$ implies that (a conjugate of) $A_{\log(n)}$ acts with a fixed point, contradicting that $A_{\log(n)}$ is a translation.  

Thus, we need only deal with the case where $\fix(B_s) \neq \emptyset$. Let $C$ be a connected component of $\R \setminus \fix(B_s)$.  For any $t$, we know that $A_t B_s A_t^{-1}$ commutes with $B_s$, so permutes the connected components of $\R \setminus \fix(B_s)$.  The family of functions $F_t:= A_t B_s A_t^{-1}$ is continuous in $t$, and $F_0(C) = C$, so we must also have $F_t(C) = C$ for all $t$.    
Now consider a connected component either of the form $(x, y)$ or of the form $(-\infty, y)$.  For sufficiently small  $t>0$, we have $y-t \in C$, so $B_s(y-t) \neq y-t$.  Thus,
$$A_t B_s A_t^{-1}(y) = B_s(y-t) + t \neq y$$ 
contradicting that $A_t B_s A_t^{-1}(C) = C$.

\end{proof}

Now we proceed with the proof of Theorem \ref{main thm}. 
Suppose for contradiction that there is a nontrivial homomorphism $\Phi: \germ \to \Homeo_+(\R)$.  Since $\germ$ is simple (Proposition \ref{simple prop}), $\Phi$ is injective.   Let $a_t$ be the germ of the translation $x \mapsto x+t$, which is an element of $\Homeo_\Z(\R)$.  Let $\Agerm = \langle a_t, b_s \rangle \subset \germ$ be the affine group constructed in Lemma \ref{affine lemma}, and 
let $I$ be a connected component of $\R \setminus \fix \Phi(a_1)$.  

Applying Militon's Lemma \ref{militon lemma} to the composition 
$$\Homeo_\Z(\R) \overset{\pi} \to \germ \overset{\Phi} \to \Homeo_+(\R)$$
we conclude that there is a homeomorphism $f: \R \to I$ such that, for all $g \in \Homeo_\Z(\R)$, the action of $\Phi(g)$ on $I$ is given by $\Phi(g)(x) = f g f^{-1}(x)$.   In particular, $\Phi(a_t)(I) = I$ holds for all $t$, and $f$ conjugates $\Phi(a_t)|_I$ to translation by $t$ on $\R$.  

Our next claim is that the elements $\Phi(b_s)$ also preserve $I$.  

\begin{lemma} \label{fixed lemma}
$\Phi(b_s)(I) = (I)$ for all $b_s \in \Agerm$.  
\end{lemma}

Let us defer the proof of Lemma \ref{fixed lemma} for a moment and see how this lemma can be used to (very quickly!) finish the proof of Theorem \ref{main thm}.

\boldhead{Proof of Theorem \ref{main thm} given Lemma \ref{fixed lemma}}
Assuming Lemma \ref{fixed lemma}, we have homeomorphisms $\Phi(a_t)|_I$ and $\Phi(b_s)|_I$ of $I$.  Conjugating by the homeomorphism $f: \R \to I$ given by Militon's lemma, $A_t:= f \Phi(a_t)|_I f^{-1}$ is translation by $t$ on $\R$, and $B_s := f \Phi(b_s)|_I f^{-1}$ is a globally defined homeomorphism of $\R$.  Moreover, $A_t$ and $B_s$ satisfy the hypotheses of Proposition \ref{germ only prop}.  But Proposition \ref{germ only prop} states that no such homeomorphisms exist.   This gives our desired contradiction.  
 \qed
 
\bigskip

It remains only to prove Lemma \ref{fixed lemma}.  

\begin{proof}[Proof of Lemma \ref{fixed lemma}]
We prove this by ``factoring" $b_s$ into a product of two germs with dynamics that we can control.  This requires a small amount of set-up.  

Define sets $S_i \subset \R$ by $S_1 : = \bigcup \limits_{n \in \Z} \left( n-\frac{1}{10}, n+\frac{1}{10} \right)$ and $S_2 := \bigcup \limits_{n \in \Z} \left( n+ \frac{4}{10}, n+\frac{6}{10} \right)$.  
Let $G_i \subset \Homeo_\Z(\R)$ be the subgroup of homeomorphisms supported on $S_i$.
\bigskip

Fix $s>0$ (the argument for $s<0$ is entirely analogous), and let $B_s$ be a homeomorphism with germ $b_s$.  Then $B_s(x) = \log(e^x + s)$ for all $x$ in some neighborhood of $\infty$.  In particular, there exists some $x_0$ such that $0 < B_s(x) - x < \frac{1}{10}$ for all $x \in [x_0, \infty)$.  
One can now easily construct a homeomorphism $f_1$ satisfying the following four properties:
\begin{enumerate}[i)]
\item $f_1(x) = x$ for $x \in S_1$
\item $f_1(x) > x$ for $x \in [x_0, \infty) \setminus S_1$
\item $f_1(x) = B_s(x)$ for $x \in S_2$
\item $f_1(x) < B_s(x)$ for $x \in [x_0, \infty) \setminus S_2$. 
\end{enumerate}
Let $f_2 = f_1^{-1} B_s$.  Thus $B_s = f_1 f_2$.  Our next goal is to show that $\Phi(f_i)(I) = I$.  

Note first that  $f_i$ is the identity on $S_i$, so $f_i$ commutes with $G_i$.  
Also, note that $f_i(x) > x$ for all $x \in [x_0, \infty) \setminus S_i$. Thus, by a straightforward generalization of Fact \ref{basic fact}, there exist continuous families of homeomorphisms $\{ h^t_1\} \subset \Homeo_+(\R)$ and $\{ h^t_2 \} \subset \Homeo_+(\R)$ for $t \in [0,1)$ such that  
\begin{enumerate}[i)]
\item $h^t_i(x) = x$ for all $x \in S_i$,
\item $h^t_i f_i (h^t_i)^{-1} \in \Homeo_\Z(\R)$, and
\item $\lim \limits_{t \to 1} h^t_i f_i (h^t_i )^{-1} = \id.$
\end{enumerate} 

By construction, $h^t_1 f_1 (h^t_1)^{-1}$ fixes $S_1$ pointwise (for every $t$), so commutes with $G_1$.  It follows that $\Phi(h^t_1 f_1 (h^t_1)^{-1})$ commutes with $\Phi(G_1)$ and so permutes the connected components of $\fix(\Phi(G_1))$.  
By Militon's Lemma, $\Phi(h^t_1 f_1 (h^t_1)^{-1})$ is a continuous family in $\Homeo_+(\R)$, with 
$$\lim \limits_{t \to 1} \Phi(h^t_1 f_1 (h^t_1)^{-1})= \id.$$  By continuity of this family (just as in the proof of Proposition \ref{germ only prop}), we conclude that $\Phi(h^t_1 f_1 (h^t_1)^{-1})$ preserves each connected component of $\R \setminus \fix(\Phi(G_1))$. 
Since $\Phi(h_1)$ also commutes with $\Phi(G_1)$, it also permutes the connected components of $\R \setminus \fix(\Phi(G_1))$, and so $\Phi(f_1)$ must preserve each connected component of $\R \setminus \fix(\Phi(G_1))$.   Militon's lemma tells us that these connected components accumulate at the endpoints of $I$, so $f_1(I) = I$.  

An identical argument can be used to show that $\Phi(f_2)(I) = I$.  Thus, $\Phi(B_s) = \Phi(f_1)\Phi(f_2)$ preserves $I$, and the lemma is proved.    This completes the proof of Theorem \ref{main thm}.

\end{proof}

\section{Applications}  \label{applications sec}

\subsection{Proof of Proposition \ref{extension prop}}
We prove Proposition \ref{extension prop} by constructing a finitely generated subgroup $\Gamma \subset \germ$ that does not extend to $\Homeo_+(\R)$.  
The strategy is similar to that of the proof of Lemma \ref{fixed lemma}, although we can no longer use Militon's lemma and continuity of the action of $\Homeo_\Z(\R)$ subgroups.  Instead, we make use of properties of extensions. 

\boldhead{Construction of $\Gamma$}
Let $S_i$ be the sets defined in Lemma \ref{fixed lemma}. 
Our group is generated by the following elements of $\germ$:
\begin{enumerate}[]
\item $t$, the germ of  $T: x \mapsto x+1$
\item $b$, the germ of $x \mapsto \log(e^x + 1)$
\item $a$, the germ of $x \mapsto x + \log(2)$
\item $f_1$ and $f_2$, where $f_i$ is the germ of a homeomorphism that fixes the set $S_i$ pointwise, satisfying $f_1f_2 = b$. (The existence of such $f_i$ follows from the proof of Lemma \ref{fixed lemma}.)
\item $g_1$ and $g_2$, germs of homeomorphisms commuting with $T$, with support contained in $S_i$.  
\end{enumerate}
Note that we have the additional relation $aba^{-1} = b^2$, that $a$ commutes with $T$, and that $g_i$ and $f_i$ commute.

\begin{claim} Let $\Gamma$ be the group generated by $t, b, a, f_1, f_2, g_1$ and $g_2$.  Then $\Gamma$ does not extend to $\Homeo_+(\R)$.
\end{claim}

\begin{proof}
Suppose for contradiction that $\Phi: \Gamma \to \Homeo_+(\R)$ is an extension.  
Assume as a first case that $\fix(\Phi(t)) = \emptyset$, so $\Phi(t)$ is conjugate to a translation.  In this case, we won't even need to consider $\Phi(f_i)$ and $\Phi(g_i)$.  
Since $\Phi(a)$ and $\Phi(t)$ commute, $\fix(\Phi(a))$ is a $\Phi(t)$-invariant set.  However, $\Phi$ is an extension, so $\Phi(a)$ has no fixed points in a neighborhood of $\infty$. Hence, $\fix(\Phi(a)) = \emptyset$.  

The relation $aba^{-1} = b^2$ (and a Banach contraction principle argument as in the proof of Proposition \ref{germ only prop}) now implies that $\fix(\Phi(b)) \neq \emptyset$.  Let $x \in \fix(\Phi(b))$.  Then 
$$\Phi(b^2 a)(x) = \Phi(a b)(x) = \Phi(a)(x)$$ so $a(x) \in \fix(\Phi(b^2)) = \fix(\Phi(b))$.  It follows that 
$\fix(\Phi(b))$ is a $\Phi(a)$-invariant set.  In particular, it contains the points $\Phi(a^n)(x)$, an unbounded sequence.  This contradicts that $\Phi$ is an extension and $b$ is a fixed point free germ.  

If instead $\fix(\Phi(t)) \neq \emptyset$, that $\Phi$ is an extension implies that $\fix(\Phi(t))$ has a rightmost point, say $x_0$.  We'll show that $\Phi(a)$ and $\Phi(b)$ both fix $x_0$.  Having shown this, the argument above applies verbatim (considering the restriction of $\Phi(a)$, $\Phi(b)$ and $\Phi(t)$ to $(x_0, \infty) \cong \R$), and gives a contradiction.  

That $\Phi(a)(x_0) = x_0$ is easy:  since $a$ and $t$ commute, $\fix(\Phi(t))$ is a $\Phi(a)$-invariant set, and in particular, its rightmost point $x_0$ must be fixed by $\Phi(a)$.   
To see that $\Phi(b)(x_0) = x_0$, we study the action of $\Phi(g_i)$.  Because $\Phi$ is an extension, there is a neighborhood of $\infty$ on which $\fix(\Phi(g_i))$ agrees with $S_i$.  Since $\Phi(g_i)$ and $\Phi(t)$ commute, $\fix(\Phi(g_i))$ is $\Phi(t)$-invariant.  Since $\Phi(t)$ is conjugate to a translation on $(x_0, \infty)$, it follows that $\fix(\Phi(g_i)) \cap (x_0, \infty)$ consists of a union of pairwise disjoint closed intervals accumulating only at $x_0$.  In other words, $x_0$ is the rightmost accumulation point of the connected components of $\fix(\Phi(g_i))$.  Since $\Phi(f_i)$ and $\Phi(g_i)$ commute, $\Phi(f_i)$ acts on $\fix(\Phi(g_i))$, and so fixes this rightmost accumulation point.  

We have just shown that $\Phi(f_i)(x_0) = x_0$.  This implies that $$\Phi(b)(x_0) = \Phi(f_1)\Phi(f_2)(x_0)=x_0,$$ which finishes the proof.  
\end{proof}

\subsection{$\germ$ does not act on the circle} \label{circle subsec}

By generalizing Lemma \ref{militon lemma}, we will prove that $\germ$ has no action on $S^1$.  
As before, let $T$ denote the central element of $\Homeo_\Z(\R)$.  Then $\Homeo_\Z(\R)/\langle T^k \rangle$ is naturally isomorphic to the subgroup $G_k \subset \Homeo_+(S^1)$ consisting of homeomorphisms that commute with an order $k$ rotation.

\begin{lemma}  \label{circle lemma}
Let $\phi: \Homeo_\Z(\R) \to \Homeo_+(S^1)$ be a homomorphism.  Either $\phi$ descends to a map $\Homeo_\Z(\R)/\langle T^k \rangle \to \Homeo_+(S^1)$ that is conjugate to the natural inclusion described above, or $\phi$ has a global fixed point and is topologically diagonal.   
\end{lemma}

\begin{proof}

It follows from the simplicity of $\Homeo_+(S^1)$ that any normal subgroup of $\Homeo_\Z(\R)$ is generated by a power of $T$.  So $\ker(\phi) = T^k$ for some $k$.   If $k>0$, then $\phi$ descends to an injective map $\bar{\phi}: \Homeo_\Z(\R)/\langle T^k \rangle \to \Homeo_+(S^1)$.    
In this case, $\bar{\phi}(T)$ is an order $k$ element of $\Homeo_+(S^1)$, hence conjugate to an order $k$ rigid rotation.  This element is central in $\bar{\phi}(\Homeo_\Z(\R)/\langle T^k \rangle)$, so (after replacing $\phi$ with a conjugate homomorphism), the image of $\bar{\phi}$ is a subgroup of $G_k$.  By Matsumoto's theorem, the induced map $\Homeo(S^1) \cong \Homeo_\Z(\R)/\langle T \rangle \to G_k/\bar{\phi}(T)$ is the standard isomorphism, so $\bar{\phi}$ is the standard inclusion.  

To treat the case of $k=0$, it suffices to prove that $\phi(\Homeo_\Z(\R)$ has a global fixed point, for we may then consider $\phi$ to have image in $\Homeo_+(\R)$ and apply Lemma \ref{militon lemma}.  
Similar to the proof of Lemma \ref{militon lemma}, we'll show that $\fix (\phi(T)) \neq \emptyset$, and that $\fix (\phi(T))$ consists of global fixed points for $\phi(\Homeo_\Z(\R))$.   

There are three possibilities for the dynamics of $\phi(T)$ (see \cite{Ghys} for a proof of this trichotomy):  
\begin{enumerate}[i)]
\item $\phi(T)$ is conjugate to an irrational rotation.
\item $\phi(T)$ has an \emph{exceptional minimal set} $K$ -- a cantor set that is contained in the closure of the orbit of a point under $\phi(T)$.  
\item $\phi(T)$ has a finite orbit
  
\end{enumerate}
In the first case, the centralizer of $\phi(T)$ is abelian, contradicting the fact that $\phi$ is injective and $\phi(\Homeo_\Z(\R))$ is contained in the centralizer of $\phi(T)$. 

In case ii), since $T$ is central, $K$ is also invariant under the action of $\phi(\Homeo_\Z(\R))$.  Collapsing the complimentary regions of $K$ to points, we get a new circle on which $\Homeo_\Z(\R)$ acts with $\phi(T)$ conjugate to an irrational rotation.  But, as we just saw above, this is impossible.  

In case iii), there is some smallest $k>0$ such that $\fix(T^k) \neq \emptyset$.  Since $T^k$ is central, $\fix(T^k)$ is a $\phi(\Homeo_\Z(\R))$-invariant set, and we get an induced action of $\Homeo_\Z(\R)/\langle T^k \rangle$ on $\fix(T^k)$.  If the action is trivial, then $k=1$ and $\fix(T)$ is a set of global fixed points.  Otherwise, it follows from our discussion of normal subgroups of $\Homeo_\Z(\R)$ (and that $k$ was minimal) that the action is faithful, in particular $\fix(T^k)$ is an infinite, circularly ordered set on which $\Homeo_\Z(\R)/\langle T^k \rangle$ acts faithfully.   This implies that $\phi$ is \emph{semi-conjugate} to a map that factors through the standard inclusion $\Homeo_\Z(\R)/\langle T^k \rangle \to \Homeo_+(S^1)$.  In particular, for any non-integer translation $\tau \in \Homeo_\Z(\R)$, the image $\phi(\tau)$ acts with no fixed point.  
 To show that the semiconjugacy is a genuine conjugacy, we need to show that $\fix(T^k) = S^1$.  If not, $S^1 \setminus \fix(T^k)$ is a collection of disjoint intervals permuted by $\phi(\Homeo_\Z(\R))$.  Since non-integer translations act without fixed points, none fixes an interval, and we conclude that there must be uncountably many disjoint intervals in $S^1 \setminus \fix(T^k)$, a contradiction.  
 \end{proof}

We can now easily conclude that $\germ$ does not act on the circle.  

\begin{proposition} \label{circle prop}
There is no nontrivial homomorphism $\germ \to \Homeo_+(S^1)$ 
\end{proposition}

\begin{proof}
Suppose $\Phi: \germ \to \Homeo_+(S^1)$ were a nontrivial homomorphism.  Since $\germ$ is simple, $\phi$ is injective.   By Lemma \ref{circle lemma}, $\Homeo_\Z(\R) \subset \germ$ maps injectively to $\Homeo_+(S^1)$, so has a global fixed point and is topologically diagonal.   The proof of Lemma \ref{fixed lemma} now goes through verbatim and shows that $\Phi(b_t)$ preserves each interval on which $\Phi(a_s)$ acts by translation, contradicting Proposition \ref{germ only prop}.  
\end{proof}

\subsection{Homomorphisms between groups of homeomorphisms} \label{militon subsec}

In \cite{Militon}, Militon proves that for any 1-manifold $M$, the only nontrivial homomorphisms $$\Homeo_c(\R) \to \Homeo(M)$$ are topologically diagonal embeddings.  As a consequence of our work, we can extend this to a statement about actions of $\Homeo_+(\R)$.  We outline the argument below.  

\begin{theorem}
Let $M$ be a 1-manifold and 
let $\phi: \Homeo_+(\R) \to \Homeo_+(M)$ be a nontrivial homomorphism.  Then $\phi$ is a topologically diagonal embedding.   
\end{theorem}

\begin{proof}
To reduce to the case of $M = \R$, we need to show that any homomorphism $\phi: \Homeo_+(\R) \to \Homeo_+(S^1)$ has a global fixed point.  Proposition 2.3 in \cite{Militon} states that the image of $\Homeo_c(\R)$ must have a fixed point.  Consider the action of $\germ = \Homeo_+(\R)/\Homeo_c(\R)$ on $\fix(\Homeo_c(\R)) \subset S^1$.   If some point in $\fix(\Homeo_c(\R))$ has finite orbit under $\germ$, then by simplicity of $\germ$ there must be a global fixed point.  Otherwise, $\fix(\Homeo_c(\R))$ is either $S^1$ or a cantor set.  In the first case, Proposition \ref{circle prop} states that the action of $\germ$ on $\fix(\Homeo_c(\R))$ is trivial.  In the second case, we can collapse the complementary regions of the cantor set to points to form a circle with an induced action of $\germ$ and apply Proposition \ref{circle prop} here to conclude that germ acts trivially on $\fix(\Homeo_c(\R))$.  

Now we need to show that any homomorphism $\phi: \Homeo_+(\R) \to \Homeo_+(\R)$ is a topologically diagonal embedding.  
We claim first that such a $\phi$ is injective.  If not, the kernel of $\phi$ is a normal subgroup, so by \cite{FS} (or by an argument very similar to our proof of Proposition \ref{simple prop}), $\ker(\phi)$ is either equal to $\Homeo_c(\R)$, to the group of homeomorphisms that pointwise fix a neighborhood of $-\infty$, or to the group of homeomorphisms that pointwise fix a neighborhood of $\infty$.  In any case, the induced map $\Homeo_+(\R)/\ker(\phi) \to \Homeo_+(\R)$ will give an injective map from either $\germ$ or $\mathcal{G}_{-\infty} \cong \germ$ to $\Homeo_+(\R)$.  But Theorem \ref{main thm} states that no such map exists.   Thus, $\phi$ is injective.

Now, by Militon's theorem in \cite{Militon}, $\phi(\Homeo_c(\R))$ is a topologically diagonal embedding.  Let $\{ I_n\}$ be the set of intervals on which the action of $\phi(\Homeo_c(\R))$ is conjugate to the standard action of $\Homeo_c(\R)$ on $\R$ via homeomorphisms $f_n: \R \to I_n$.  
Since $\Homeo_c(\R) \subset \Homeo_+(\R)$ is normal, for any $g \in \Homeo_+(\R)$, the map $\phi(g)$ permutes the intervals $I_n$.   
As we did in the proofs of Proposition \ref{simple prop} and Theorem \ref{main thm} we can now use continuity to show that $\phi(g)(I_n) = I_n$.  In more detail, one can factor $g$ as a finite product $g = f g_1 g_2 ... g_k$ where $f \in \Homeo_c(\R)$ and each $g_i$ lies in some conjugate of $\Homeo_\Z(\R)$.  Since the restriction of $\phi$ to $\Homeo_\Z(\R)$ is continuous, as is its restriction to $\Homeo_c(\R)$, we can build a path $g_t$ from $g$ to the identity such that $\phi(g_t)$ is continuous in $t$.  Each $\phi(g_t)$ permutes the intervals $I_n$, so by continuity $\phi(g)(I_n) = I_n$.  

It remains only to show that the restriction of $\phi(g)$ to $I_n$ agrees with $f_n g f_n^{-1}$ for all $g \in \Homeo_+(\R)$.  We already know this is true for any element $g \in \Homeo_c(\R)$.  To see this for general $g$, let $x \in I_n$ and consider a sequence $h_k \in \Homeo_c(\R)$ with $\cap_k \supp(h_k) = f_n^{-1}(x)$.  (Here $\supp(h_k)$ denotes the support of $h_k$).  
Then $\cap_k \supp(\phi(h_k)) = x$, and 
$$\phi(g) (x) = \phi(g)\left( \bigcap \limits_k \supp(\phi(h_k)) \right) = \bigcap \limits_k \supp(\phi(g h_k g^{-1})) = f_n \left( \bigcap \limits_k (\supp(g h_k g^{-1}) \right)$$
but  $f_n \left( \bigcap \limits_k (\supp(g h_k g^{-1}) \right) = f_n (g f_n^{-1}(x))$, and this is what we needed to show.  

\end{proof}

\section{Other left-orderable groups that don't act on the line}  \label{Examples sec}

We conclude by illustrating a different approach to construct left-orderable groups that don't act on the line, inspired by C. Rivas.  
In this approach, one takes a group $\Gamma$ which has very few left orders (or equivalently, very few actions on the line) and builds a group $G$ containing uncountably many copies of $\Gamma$.   The goal is to define appropriate relations between the copies of $\Gamma$ so as to force any action of $G$ on the line to be supported on uncountably many \emph{disjoint} intervals -- which is, of course, impossible.

To illustrate the technique, we begin with a quick example of a left-orderable group of cardinality $|\R|$ that has no dynamical realization.  

\begin{proposition}
For each $r \in \R$, let $G_r \cong \Homeo_\Z(\R)$.  Let $G$ be the (external) direct product of the $G_r$.  Then $G$ is a left-orderable group of cardinality $|\R|$ that has no faithful action on the line.  
\end{proposition}

\begin{proof}
$G$ is left orderable since it is the direct product of left-orderable groups, and of cardinality $|\R|$ since it is generated by continuum-many groups of cardinality $|\R|$.   Suppose now for contradiction that $\phi: G \to \Homeo_+(\R)$ is an injective homomorphism.  Then by Lemma \ref{militon lemma}, for any $r, s \in \R$, the images $\phi(G_r)$ and $\phi(G_s)$ are commuting, topologically diagonal embeddings of $\Homeo_\Z(\R)$.  It follows easily that $\phi(G_r)$ and $\phi(G_s)$ are supported on disjoint intervals (see Lemma 4.1 in \cite{Militon}).   
Thus, $\{ \supp(\phi(G_r) \mid r \in \R \}$ is a collection of uncountably many pairwise disjoint sets in $\R$, each with nonempty interior, a contradiction. 

\end{proof}

Producing a group with no action on $\R$ whatsoever takes a bit more work.   The example below is due to Rivas \cite{Rivas}.  Instead of $\Homeo_\Z(\R)$, Rivas' construction uses the Klein bottle group $K := \langle a,b\mid aba^{-1}=b^{-1}\rangle$, which also has very few actions on the line.  (To be precise, $K$ admits only four left-orderings, and only two faithful actions on the line up to semi-conjugacy in $\Homeo(\R)$, but this fact is not used in the proof.  See Theorem 5.2.1 in \cite{kopytov}.)

\begin{proposition}[Rivas] \label{Rivas ex}
Let $G$ be the group generated by $\{a_s \mid s \in \R\}$ with relations 
$$a_t a_s a_t^{-1} = a_s^{-1} \text{ if } t < s.$$
Then $G$ is left-orderable, but has no action on the line.  
\end{proposition}

\begin{proof}
To see that $G$ is left-orderable is not difficult.  To be consistent with our earlier work, we'll give a proof using Proposition \ref{lo criterion}, starting with an easy criterion to show an element of $G$ is nontrivial.  Given $g= a_{s_1}^{n_1}a_{s_2}^{n_2} \ldots a_{s_k}^{n_k} \in G$, let $s = \min{s_i}$ and consider the sum of the exponents $n_k$ over all $k$ such that $s_k = s$.  Call this sum $\tau(g)$.  It follows from the definition of $G$ that $g \neq \id$ whenever $\tau(g)$ is nonzero.  

Given a finite collection $g_1, ..., g_n$ of nontrivial elements, define $\epsilon_i = 1$ if $\tau(g) > 0$, and $\epsilon_i = -1$ if $\tau(g) < 0$.   It follows that for any word $w$ in the semigroup generated by $\{ g_1^{\epsilon_1}, ... , g_n^{\epsilon_n}\}$, we will have $\tau(w) > 0$; in particular $w \neq \id$.  

To show that $G$ has no action on $\R$, we start with a quick lemma about $K$.  

\begin{lemma}\label{klein lemma}
Let $K= \langle a,b\mid aba^{-1}=b^{-1}\rangle$, and let $\phi: K \to \Homeo_+(\R)$ be a homomorphism such that $\phi(b) \neq \id$.  Let $I$ be any connected component of $\R \setminus \fix(\phi(b))$.  Then $\phi(a)(I) \cap I = \emptyset$.  
\end{lemma}

\begin{proof}
Since $\langle b \rangle \subset K$ is a normal subgroup, $\phi(a)$ permutes the connected components of $\R \setminus \fix(\phi(b))$.  
Thus, either $\phi(a)(I) = I$ or $\phi(a)(I) \cap I = \emptyset$.  Since $\phi(b)$ fixes no point in $I$, the restriction of $\phi(b)$ to $I$ is conjugate to a translation.  If $\phi(a)(I) = I$, then $\phi(a)|_I$ is an orientation-preserving homeomorphism of $I$ conjugating the translation $\phi(b)|_I$ to its inverse, which is impossible.  
\end{proof}

Suppose now for contradiction that there is a nontrivial homomorphism $\phi: G \to \Homeo_+(\R)$.  In particular, $\phi(a_s)$ is nontrivial for some $s$. Let $I_s$ be a connected component of $\R \setminus \fix(a_s)$.  

Consider any $r < t< s$.  We claim that $\phi(a_t)(I_s) \cap \phi(a_r)(I_s) = \emptyset$.  To see this, first note that the subgroup of $G$ generated by $a_s$ and $a_t$ is isomorphic to $K$, and Lemma \ref{klein lemma} implies that $\phi(a_t)(I_s) \cap I_s = \emptyset$.   From this, it follows also that $I_s \cup \phi(a_t)(I_s)$ is properly contained in some connected component $I_t$ of $\R \setminus \fix(\phi(a_t))$.   The subgroup generated by $a_t$ and $a_r$ is also isomorphic to $K$, and so Lemma \ref{klein lemma} implies that 
 $\phi(a_r)(I_t) \cap I_t = \emptyset$ holds as well.  It follows that  $\phi(a_t)(I_s) \subset I_t$ and $ \phi(a_r)(I_s) \subset \phi(a_r)(I_t)$ are disjoint.  
We conclude that $\{ \phi(a_t)(I_s) \mid t<s \}$ is an uncountable collection of pairwise disjoint open intervals in $\R$, which is absurd.

\end{proof}


\vspace{1in}

Dept. of Mathematics, UC Berkeley 

970 Evans Hall \#3840

Berkeley, CA 94720

E-mail: kpmann@math.berkeley.edu

\end{document}